\documentclass[12pt,a4paper,oneside,reqno]{amsproc}
\usepackage{amsmath}
\usepackage{amssymb,amsfonts,mathrsfs}
\usepackage[colorlinks=true,linkcolor=blue,citecolor=blue]{hyperref}
\usepackage[driver=pdftex,lmargin=2.5cm,rmargin=2.5cm,heightrounded=true,centering]{geometry}
\usepackage{math50}
\usepackage{local}
\usepackage{epsfig}
\usepackage{graphicx}
\usepackage{xy}
\xyoption{all}
\def\XXint#1#2#3{{\setbox0=\hbox{$#1{#2#3}{\int}$}
     \vcenter{\hbox{$#2#3$}}\kern-.5\wd0}}

\begin{document}
\title[$Q_{\gamma}$ curvature problem on $\mathbb{S}^n$]
{A perturbation result for the $Q_{\gamma}$ curvature problem on $\mathbb{S}^n$}
\author{Guoyuan Chen}
\address{School of Mathematics and Statistics, Zhejiang University of Finance \& Economics, Hangzhou 310018, Zhejiang, P. R. China}
\email{gychen@zufe.edu.cn}
\author{Youquan Zheng}
\address{School of Science, Tianjin University, Tianjin 300072, P. R. China.}
\email{zhengyq@tju.edu.cn}
\newcommand{\optional}[1]{\relax}
\setcounter{secnumdepth}{3}
\setcounter{section}{0} \setcounter{equation}{0}
\numberwithin{equation}{section}
\newcommand{\MLversion}{1.1}
\date{\today}
\begin{abstract}
We consider the problem of prescribing the $Q_{\gamma}$ curvature on $\mathbb{S}^n$.
Using a perturbation method, we obtain existence results for curvatures close to a positive constant.
\end{abstract}
\maketitle

\section{Introduction}
Let $(M, g_0)$ be a $C^\infty$ compact Riemannian manifold of dimension $n \geq 3$. The famous Yamabe problem concerns the existence of a metric conformal to $g_0$ with constant scalar curvature. This corresponds to solve the following partial differential equation
\begin{equation}\label{e:Yamabe}
-\Delta_{g_0}v + \frac{n - 2}{4(n-1)}R_{g_0}v = \frac{n - 2}{4(n - 1)}R_g v^{\frac{n + 2}{n - 2}},\quad v > 0,
\end{equation}
where $R_g \equiv constant$ is the scalar curvature of $g$.

The linear operator which appears as the first two terms on the left of (\ref{e:Yamabe}) is known as the conformal Laplacian associated to the metric $g_0$ and denoted as $P_1^{g_0}$. It is conformally covariant in the sense that if $f$ is any smooth function and $g = v^{\frac{4}{n - 2}}g_0$ for some $v > 0$, then
\begin{equation}\label{e:conformallycovariant}
P_1^{g_0}(vf) = v^{\frac{n + 2}{n - 2}}P_1^{g}(f).
\end{equation}
Setting $f \equiv 1$ in (\ref{e:conformallycovariant}) yields the familiar relationship (\ref{e:Yamabe}) between the scalar curvatures $R_{g_0}$ and $R_g$.
Another conformally covariant operator is
\begin{equation*}\label{e:Paneitz}
P_2^g = (-\Delta_g)^2 - {\rm div}_g (a_n R_gg + b_nRic_g)d + \frac{n -4}{2}Q_n^g,
\end{equation*}
which was discovered by Paneitz in the 1980s, see \cite{Paneitz2008} and \cite{DjadliHebeyLedoux2000}. Here $Q_n^g$ is the standard $Q$-curvature, $Ric_g$ is the Ricci curvature of $g$ and $a_n$, $b_n$ are constants depending on $n$.

$P_1$ and $P_2$ (in the following, the superscript $g$ is omitted when there is no ambiguity) are the first two of a sequence of conformally covariant elliptic operators, $P_k$, which exist for all $k\in \mathbb{N}$ if $n$ is odd, but only for $k\in\{1,\cdot\cdot\cdot, n/2\}$ if $n$ is even.
The first construction of these operators was given by Graham-Jenne-Masion-Sparling in \cite{GrahamJenneSparling1992}. In that paper, the authors used the ambient metric construction of Fefferman and Graham systematically to construct conformally invariant powers of the Laplacian. In odd dimensions this construction is unobstructed but for dimension $n = 2m$ gives an invariant form of $(-\Delta)^k$ only for $k\leq m$. In \cite{GrahamCRobinJLMS1992}, Graham showed that this result cannot be improved in four dimension.

This leads naturally to the question whether there exist any conformally covariant pseudodifferential operators of noninteger orders. In \cite{Peterson2000}, the author constructed an intrinsically defined conformally covariant pseudo-differential operator of arbitrary real-number order acting on scalar functions. In the work of Graham and Zworski \cite{GramZworski2003}, they proved $P_k$ can be realized as residues at the values $\gamma = k$ of a meromorphic family of scattering operators. Using this opinion, one obtains a holomorphic family of elliptic pseudodifferential operators $P_{\gamma}^{g}$ for noninteger $\gamma$. We will recall this definition in Section 2. An alternative construction of these operators has been obtained by Juhl in \cite{Juhl2009} and \cite{Juhl2010}.

In recent years, there are extensive works on the properties of fractional Laplacian operators as non-local operators together with applications to free-boundary value problems and non-local minimal surfaces, for example, \cite{Caffarelli&Silvestre07}, \cite{CaffarelliRoquejoffreSavin}, \cite{CaffarelliSalsa&Silvestre}, \cite{CaffarelliValdinoci2011}, \cite{CaffarelliVasseurDFDQGE2010} and so on. Mathematically, $(-\Delta)^\gamma$ is defined as
$$
(-\Delta)^\gamma u = C(n, \gamma)\mbox{P.V.} \int_{\mathbb{R}^n}\frac{u(x) - u(y)}{|x - y|^{n + 2\gamma}}dy = C(n, \gamma)\lim_{\varepsilon\to 0^+}\int_{B_{\varepsilon}^c(x)}\frac{u(x) - u(y)}{|x - y|^{n + 2\gamma}}dy.
$$
Here P.V. is a commonly used abbreviation for `in the principal value sense' and $C(n, \gamma) = \pi^{-(2\gamma + n/2)}\frac{\Gamma(n/2 + \gamma)}{\Gamma(-\gamma)}$.
It is well known that $(-\Delta)^\gamma$ on $\mathbb{R}^{n}$ with $\gamma\in (0, 1)$ is a nonlocal operator. In the remarkable work of Caffarelli and Silvestre \cite{Caffarelli&Silvestre07}, the authors express this nonlocal operator as a generalized Dirichlet-Neumann map for a certain elliptic boundary value problem with local differential operators defined on the upper half-space $\mathbb{R}^{n+1}_{+} = \{(x, t): x\in\mathbb{R}^{n}, t > 0\}$. That is, given a solution $u = u(x)$ of $(-\Delta)^\gamma u = f$ in $\mathbb{R}^{n}$, one can equivalently consider the dimensionally extended problem for $u = u(x, t)$, which solves
\begin{equation*}
\begin{cases} \text{div}(t^{1 -2\gamma}\nabla u) = 0 , &  \text{ in }\mathbb{R}^{n+1}_{+},\\
-d_{\gamma}t^{1 -2\gamma}\partial_{t}u|_{t\to 0} = f, & \text{ on $\partial\mathbb{R}^{n+1}_{+}$}.
\end{cases}
\end{equation*}
Here the positive constant $d_{\gamma} > 0$ is explicitly given by
$$
d_{\gamma} = 2^{2\gamma -1}\frac{\Gamma(\gamma)}{\Gamma(1 - \gamma)}.
$$
In the work of Chang and Gonzalez \cite{ChangGonzalez2011}, they extended the work of \cite{Caffarelli&Silvestre07} and characterized $P_{\gamma}$ as such a Dirichlet-to-Neumann operator on a conformally compact Einstein manifold.

We focus only on the operators $P_{\gamma}$ when $\gamma\in \mathbb{R}$ and $|\gamma|\leq \frac{n}{2}$. These operators have the following conformally covariant properties: 
if $g = v^{\frac{4}{n - 2\gamma}}g_0$, then
\begin{equation}\label{e:conformalltransform}
P_{\gamma}^{g_0}(vf) = v^{\frac{n + 2\gamma}{n - 2\gamma}}P_{\gamma}^g(f)
\end{equation}
for any smooth function $f$. Generalizing the formula for scalar curvature and the Paneitz-Branson $Q$-curvature, the $Q$-curvature for $g$ of order $\gamma$ is defined as
\begin{equation*}
Q_{\gamma}^{g} = P_{\gamma}^{g}(1).
\end{equation*}

Thus one can consider the `fractional Yamabe problem': given a metric $g_0$ on a compact manifold $M$, find a function $v > 0$ on $M$ such that if $g = v^{\frac{4}{n - 2\gamma}}g_0$, then $Q_{\gamma}^{g}$ is constant. By (\ref{e:conformalltransform}), one has  to solve the equation
\begin{equation*}
P_{\gamma}^{g_0} v = Q_{\gamma}^{g} v^{\frac{n + 2\gamma}{n - 2\gamma}}, \,\,\, v > 0 {\rm \,\,\,on\,\,\,} M
\end{equation*}
with $Q_{\gamma}^{g} \equiv constant$. In this direction, we refer the interested readers to the papers \cite{GonzalezAPDE2013}, \cite{GonzalezJGA2012}, \cite{QingRaske2006}, \cite{JinXiongfYfsapplications2013}, \cite{JinLiyanyanXiongNirenbergproblemBlowupanalysis} and the references therein.

Also, one has the `prescribing $Q_\gamma$ curvature problem':
given a metric $g_0$ on a compact manifold $M$ and a smooth function $Q$ on $M$, find $v > 0$ so that if $g = v^{\frac{4}{n - 2\gamma}}g_0$, then $Q_{\gamma}^{g} = Q$. By (\ref{e:conformalltransform}), this amounts to solve
\begin{equation*}
P_{\gamma}^{g_0} v = Qv^{\frac{n + 2\gamma}{n - 2\gamma}}, \,\,\, v > 0 {\rm \,\,\,on\,\,\,} M.
\end{equation*}

In this paper, we consider the prescribing $Q_\gamma$ curvature problem on $\mathbb{S}^n$,
\begin{equation}\label{e:Qcurvaturesphere}
P_{\gamma}^{g_c} v = Qv^{\frac{n + 2\gamma}{n - 2\gamma}}, \,\,\, v > 0{\rm\,\,\, on\,\,\,}\mathbb{S}^n.
\end{equation}
Here $g_c$ denotes the canonical metric on $\mathbb{S}^n$. Let $\mathcal{N}$ be the north pole of $\mathbb{S}^n$ and
\begin{equation*}
F:\mathbb{R}^n\to \mathbb{S}^n\setminus\{\mathcal{N}\},\quad x\mapsto \left(\frac{2x}{1 + |x|^2}, \frac{|x|^2 - 1}{|x|^2 + 1}\right)
\end{equation*}
be the inverse of stereographic from $\mathbb{S}^n\setminus\{\mathcal{N}\}$ to $\mathbb{R}^n$. Then
\begin{equation*}\label{e:Qcurvature}
P_{\gamma}^{g_c}(\phi)\circ F = |J_F|^{-\frac{n + 2\gamma}{2n}}(-\Delta)^{\gamma}(|J_F|^{\frac{n - 2\gamma}{2n}}(\phi\circ F)){\rm\,\,\, for\,\, all\,\,\,}\phi\in C^{\infty}(\mathbb{S}^n),
\end{equation*}
where $|J_F| = \left(\frac{2}{1 + |x|^2}\right)^n$ and $(-\Delta)^{\gamma}$ is the fractional Laplacian operator. Hence
$u(x) = |J_F|^{\frac{n - 2\gamma}{2n}}v(F(x))$ satisfies the following $Q_{\gamma}$ curvature problem,
\begin{equation}\label{e:main}
(-\Delta)^\gamma u = (Q\circ F) u^{\frac{n + 2\gamma}{n - 2\gamma}},\,\,\, u > 0 {\rm\,\,\,on\,\,\,}\mathbb{R}^n.
\end{equation}

We assume that $Q\circ F$ is of form
\begin{equation*}
Q\circ F = 1 + \varepsilon K(x),
\end{equation*}
where $\varepsilon$ is a small real number and $K(x)$ satisfies the following conditions:
\begin{enumerate}
\item[(K1)] $K\in L^{\infty}(\mathbb{R}^n)\cap C^2(\mathbb{R}^n)$ and $\exists\,\, \eta > 0$ such that
$${\rm when}\,\,|x|\geq \eta, \quad\langle K'(x), x\rangle < 0,$$
and
$$\langle K'(x), x\rangle\in L^1(\mathbb{R}^n), \quad\displaystyle\int_{\mathbb{R}^n}\langle K'(x), x\rangle dx < 0;$$
\item[(K2)] $K$ has finitely many critical points (denote the set of critical points as Crit($K$));
\item[(K3)] $\forall\xi\in {\rm Crit}(K)$,
there exist $\beta = \beta_\xi\in (1, n)$ and a function $Q_y: \mathbb{R}^n\to \mathbb{R}$, depending continuously on $y$ locally near $\xi$, such that $A_\xi: = \frac{1}{p + 1}\displaystyle\int_{\mathbb{R}^n}Q_\xi(y)z_0^{p + 1}(y)dy \neq 0$ and
\begin{eqnarray*}
&&Q_y(\lambda x) = \lambda^\beta Q_y(x), \quad \forall \lambda \geq 0,\\
&&K(x) = K(y) + Q_\xi(x - y) + o(|x - y|^\beta),{\rm\,\,\, as\,\,\,} x\to y;
\end{eqnarray*}
\item[(K4)]$\displaystyle\sum_{A_\xi < 0}{\rm deg_{loc}}(K', \xi)\neq (-1)^n$.
\end{enumerate}

Our main theorem is
\begin{theorem}\label{t:main}
Suppose $\gamma\in (0, n/2)$ and $K$ satisfies $(K1)$-$(K4)$. Then for $\varepsilon$ small, problem (\ref{e:main}) has a solution $u_\varepsilon$ of class $C^\infty$ and $\lim_{|x|\to 0}u(x) = 0$. Equivalently, the $Q_\gamma$ curvature problem (\ref{e:Qcurvaturesphere}) is solvable when $Q$ is of form $1 + \varepsilon K\circ F^{-1}$ and $\varepsilon$ is small.
\end{theorem}

As a corollary, we have
\begin{cor}\label{c:corollary}
Suppose $\gamma\in (0, n/2)$, $K$ satisfies $(K1)$, $(K2)$ and one of the following conditions:
\begin{enumerate}
\item[(K5)]$\Delta K(\xi) \neq 0$ for each $\xi\in {\rm Crit}(K)$ and
\begin{equation*}
\displaystyle\sum_{\Delta K(\xi) < 0}{\rm deg_{loc}}(K', \xi) \neq (-1)^n.
\end{equation*}
\item[(K6)] $\forall \xi\in {\rm Crit}(K)$, $\exists \beta\in (1, n)$ and $a_j\in C(\mathbb{R}^n)$, with $\tilde{A}_\xi: = \sum a_j(\xi)\neq 0$ and such that $K(x) = K(\eta) + \sum a_j|x - \eta|^\beta + o(|x - \eta|^\beta)$ as $x\to \eta$, for all $\eta$ locally near $\xi$. Moreover there results
    \begin{equation*}
    \displaystyle\sum_{\tilde{A}_\xi < 0}{\rm deg_{loc}}(K', \xi) \neq (-1)^n.
    \end{equation*}
\end{enumerate}
Then for $\varepsilon$ small, problem (\ref{e:main}) has a solution $u_\varepsilon$ of class $C^\infty$ and $\lim_{|x|\to 0}u(x) = 0$. Equivalently, the $Q_\gamma$ curvature problem (\ref{e:Qcurvaturesphere}) is solvable when $Q$ is of form $1 + \varepsilon K\circ F^{-1}$ and $\varepsilon$ is small.
\end{cor}

When $\gamma = 1$, Theorem \ref{t:main} and Corollary \ref{c:corollary} was proved in \cite{AmbrosettiGarciaPeral}. Condition $(K5)$ was essentially the case handled in \cite{BahriCoronJFA1991} and $(K6)$ was discussed in \cite{LiyanyanPSCRPJDE1995} and \cite{LiYanyanPSCRPCPAM1996}.
To prove the main theorem, we will use a perturbed method. For this method and its applications, we refer the reader to
\cite{AmbrosettiGarciaPeral},\cite{MalchiodiUguzzoni2002}, \cite{AmbrosettiMalchiodi2006} and the references therein.

This paper is organized as follows: in Section 2, we will briefly describe the work of Graham and Zworski \cite{GramZworski2003} and the notation of the fraction Paneitz operator $P_{\gamma}$. A abstract perturbation theorem will be included in Section 3. In Section 4-6, we will verify the conditions of the abstract theorem to obtain a positive smooth solution.

\subsection*{Acknowledgements}
The first author was partially supported by Zhejiang Provincial Natural Science Foundation of
China (LQ13A010003). The second author was partially supported by NSFC of China (11271200).
Part of this work was done when the first author visited the Chern Institute of Mathematics,
Nankai University; he is very grateful to the institution for the kind hospitality. 

\section{Fractional Paneitz operators and the $Q_{\gamma}$ curvature.}
In this section, we briefly recall the definition of the fractional Paneitz operators and the $Q_{\gamma}$ curvature. The main references are \cite{ChangGonzalez2011}, \cite{GramZworski2003} and \cite{FeffermanGraham2012}.

Let $M$ be a compact manifold of dimension $n$ with a metric $\hat{g}$. Suppose that $X^{n + 1}$ is a smooth manifold with boundary $\partial X^{n + 1} = M$. A defining function $\rho$ of the boundary $M$ in $X$ is a smooth function on $X$ such that
\begin{enumerate}
\item[(1)] $\rho > 0$ {\rm\,\,\,in\,\,\,} $X$;
\item[(2)] $\rho = 0$ {\rm\,\,\,on\,\,\,} $M$;
\item[(3)] $d\rho\neq 0$ {\rm\,\,\,on\,\,\,} $M$.
\end{enumerate}
A complete Riemannian metric $g$ on $X$ is said to be conformally compact
of regularity $C^{k, \alpha}$ if $\rho^2g$ extends to be a $C^{k, \alpha}$ compact Riemannian metric on
$\overline{X}$ for a defining function $\rho$ of the boundary $M$ in $X$.

Given a conformally compact, asymptotically hyperbolic manifold $(X^{n + 1}, g^+)$ and a representative $\hat{g}$ in $[\hat{g}]$ on the conformal infinity $M$, there is a uniquely defining function $\rho$ such that, on $M\times (0, \delta)$ in $X$, $g^+$ has the normal form $g^+ = \rho^{-2}(d\rho^2 + g_\rho)$, where $g_\rho$ is a one-parameter family of metrics on $M$ satisfying $g_\rho|_M = \hat{g}$.
From \cite{GramZworski2003} or \cite{MazzeoMelrose1987}, we know that given $f\in C^{\infty}(M)$ and $s\in \mathbb{C}$, the eigenvalue problem
\begin{equation*}
-\Delta_{g^+} u - s(n - s)u = 0,\quad\rm{in\,\,}X
\end{equation*}
has a solution of the form
\begin{equation*}
u = F\rho^{n - s} + H\rho^s,\quad F, H\in C^{\infty}(X),\quad F|_{\rho = 0} = f,
\end{equation*}
for all $s\in \mathbb{C}$ unless $s(n-s)$ belongs to the pure point spectrum of $-\Delta_{g^+}$. Then, the scattering operator on $M$ is defined as $S(s)f = H|_M$. It is a meromorphic family of pseudo-differential operators in $Re(s) > n/2$. The values $s = n / 2, n/2 + 1, n/2 + 2$, $\cdots$ are simple poles of finite rank, these are known as the trivial poles; $S(s)$ may have other poles.

Then one can define the fractional Paneitz operators as follows: for $s = n/2 + \gamma$, $\gamma\in (0, n/2)$, $\gamma\not\in \mathbb{N}$, set
\begin{equation*}
P_{\gamma}[g^+, \hat{g}]:=d_{\gamma}S(n/2 + \gamma), \quad d_{\gamma} = 2^{2\gamma}\frac{\Gamma(\gamma)}{\Gamma(-\gamma)}.
\end{equation*}
The principal symbol of $P_{\gamma}$ is equal to the principal symbol of the fractional Laplacian $(-\Delta_{\hat{g}})^{\gamma}$. Thus $P_{\gamma} = (-\Delta_{\hat{g}})^{\gamma} + \Psi_{\gamma - 1}$, where $\Psi_{\gamma - 1}$ is a pseudo-differential operator of order $\gamma - 1$.

For a conformal change of metric
\begin{equation*}\label{conformaltransform}
\hat{g}_{v} = v^{\frac{4}{n - 2\gamma}}\hat{g}, \quad v > 0,
\end{equation*}
$P_{\gamma}[g^+, \hat{g}]$ satisfies
\begin{equation*}
P_{\gamma}[g^+, \hat{g}_{v}]\phi = v^{-\frac{n + 2\gamma}{n - 2\gamma}}P_{\gamma}[g^+, \hat{g}](v\phi),
\end{equation*}
for all smooth functions $\phi$, see \cite{GramZworski2003}.
The $Q_{\gamma}$ curvature is defined to be
\begin{equation*}
Q_{\gamma}[g^+, \hat{g}]: = P_{\gamma}[g^+, \hat{g}](1).
\end{equation*}
Therefore, we have
\begin{equation*}\label{Qgammacurvature}
P_{\gamma}[g^+, \hat{g}]v = v^{\frac{n + 2\gamma}{n - 2\gamma}}Q_{\gamma}[g^+, \hat{g}_v].
\end{equation*}

When $\gamma = k\in \mathbb{N}$, $P_k$ are the conformally invariant powers of the Laplacian constructed in the classical paper by Graham, Jenne, Mason and Sparling \cite{GramZworski2003}, or Fefferman and Graham \cite{FeffermanGraham2012}. They are local operators and satisfy
\begin{equation*}
P_k = (-\Delta)^k + \rm{\,\,lower\,\, order\,\, terms}.
\end{equation*}
For examples, when $k = 1$, it is the conformal Laplacian,
\begin{equation*}
P_1 = -\Delta_g + \frac{n - 2}{4(n - 1)}R_g,
\end{equation*}
and when $k = 2$, it is the Paneitz operator (\cite{Paneitz2008} and \cite{DjadliHebeyLedoux2000})
\begin{equation*}
P_2 = (-\Delta_g)^2 - {\rm div}_g (a_n R_gg + b_nRic_g)d + \frac{n -4}{2}Q_n^g.
\end{equation*} 

\section{The abstract perturbation method.}
In this section, we recall some abstract perturbation results, variational in nature, developed in \cite{AmbrosettiBadiale1998} and \cite{AmbrosettiCotiZelatiEkeland1987}, which we will use to obtain our existence results.

Let $E$ be a Hilbert space and let $f_0$, $G\in C^2(E, \mathbb{R})$ be given. Consider the perturbed functional
\begin{equation*}
f_\varepsilon (u) = f_0(u) - \varepsilon G(u).
\end{equation*}

Suppose that $f_0$ satisfies:
\begin{enumerate}
\item[(i)] $f_0$ has a finite dimensional manifold $Z$ of critical points; we assume that $Z$ is parameterized by a function $\alpha: A\to Z$, $A$ being an open subset of $\mathbb{R}^d$, $d\geq 1$;
\item[(ii)] for all $z\in Z$, $D^2f_0(z)$ is a Fredholm operator with index zero;
\item[(iii)] for all $z\in Z$, there holds $T_zZ = {\rm ker} D^2f_0(z)$.
\end{enumerate}

We define $\Gamma: A\to \mathbb{R}$ as $\Gamma = G\circ \alpha$. The following theorem was proved in \cite{AmbrosettiBadiale1998}, see Theorem 2.1 and Remark 2.2 in \cite{AmbrosettiGarciaPeral}.
\begin{theorem}\label{t:abstracttheorem}
Let $f_0$ satisfies (i)-(iii) above and suppose that there exists an open subset $\Omega\subseteq A$ such that $\Gamma'\neq 0$ on $\partial\Omega$ and
\begin{equation}\label{e:degreformula3}
{\rm deg}(\Gamma', \Omega, 0)\neq 0
\end{equation}
Then for $|\varepsilon|$ small enough there exists a critical point $u_\varepsilon$ of $f_\varepsilon$.
\end{theorem}

The inclusion $T_zZ \subseteq {\rm ker} D^2f_0(z)$ is always true and condition (iii) is a non-degeneracy condition which allows to apply the Implicit Function Theorem. The solution $u_\varepsilon$ is close to $Z$ in the sense that there exists $z_\varepsilon \in Z$, such that $\|u_\varepsilon - z_\varepsilon\|\leq C \varepsilon$, for some constant $C$ depending on $\Omega$, $f_0$ and $G$.

\section{The variational space}
In this section, we first recall some useful facts of the fractional order Sobolev spaces, then we choose the variational space that is suitable for our argument. For more details, please see, for example, \cite{AdamsSobolevSpace}, \cite{Shubin01}, \cite{NezzaPalatucci&Valdinoci}.

Consider the Schwartz space $\mathcal{S}$ of rapidly decaying $C^{\infty}$ functions on $\mathbb{R}^{n}$. The topology of this space is generated by the seminorms
$$
p_{N}(\varphi) = \displaystyle\sup_{x\in\mathbb{R}^{n}}(1 + |x|)^{N}\sum_{|\alpha|\leq N}|D^{\alpha}\varphi(x)|,\quad N = 0, 1, 2,\cdot\cdot\cdot,
$$
where $\varphi\in \mathcal{S}$. Let $\mathcal{S}'$ be the set of all tempered distributions, which is the topological dual of $\mathcal{S}$. As usual, for any $\varphi\in \mathcal{S}$, we denote by
$$
\mathcal{F}\varphi(\xi) = \frac{1}{(2\pi)^{n/2}}\int_{\mathbb{R}^{n}}e^{-i\xi\cdot x}\varphi(x)dx
$$
the Fourier transformation of $\varphi$ and we recall that one can extend $\mathcal{F}$ from $\mathcal{S}$ to $\mathcal{S}'$.

When $\gamma\in (0, 1)$, the space $H^{\gamma}(\mathbb{R}^{n}) = W^{\gamma, 2}(\mathbb{R}^n)$ is defined by
\begin{eqnarray*}
H^{\gamma}(\mathbb{R}^{n})& = &\left\{u\in L^2(\mathbb{R}^2): \frac{|u(x) - u(y)|}{|x - y|^{n/2 + \gamma}}\in L^{2}(\mathbb{R}^n\times\mathbb{R}^n)\right\}\\
& = & \left\{u\in L^2(\mathbb{R}^2): \int_{\mathbb{R}^n}(1 + |\xi|^{2\gamma})|\mathcal{F}u(\xi)|^2d\xi < +\infty\right\}
\end{eqnarray*}
and the norm is
\begin{eqnarray*}
\|u\|_{\gamma} := \|u\|_{H^{\gamma}(\mathbb{R}^{n})}& = &\left(\int_{\mathbb{R}^n}|u|^2dx + \int_{\mathbb{R}^n}\int_{\mathbb{R}^n}\frac{|u(x) - u(y)|^2}{|x - y|^{n + 2\gamma}}dxdy\right)^{1/2}.
\end{eqnarray*}
Here the term
$$
[u]_{\gamma} := [u]_{H^{\gamma}(\mathbb{R}^{n})} = \left(\int_{\mathbb{R}^n}\int_{\mathbb{R}^n}\frac{|u(x) - u(y)|^2}{|x - y|^{n + 2\gamma}}dxdy\right)^{1/2}
$$
is the so-called Gagliardo (semi) norm of $u$. The following identity yields the relation between the fractional operator $(-\Delta)^\gamma$ and the fractional Sobolev space $H^{\gamma}(\mathbb{R}^{n})$,
$$
[u]_{H^{\gamma}(\mathbb{R}^{n})} = C\left(\int_{\mathbb{R}^n}|\xi|^{2\gamma}|\mathcal{F}u(\xi)|^2d\xi\right)^{1/2} = C\|(-\Delta)^{\gamma/2}u\|_{L^2(\mathbb{R}^n)}
$$
for a suitable positive constant $C$ depending only on $\gamma$ and $n$.

When $\gamma > 1$ and it is not an integer we write $\gamma = m + \sigma$, where $m$ is an integer and $\sigma\in (0, 1)$. In this case the space $H^{\gamma}(\mathbb{R}^{n})$
consists of those equivalence classes of functions $u\in H^{m}(\mathbb{R}^{n})$ whose distributional derivatives $D^{J} u$, with $|J| = m$, belong to $H^{\sigma}(\mathbb{R}^{n})$, namely
\begin{eqnarray*}
H^{\gamma}(\mathbb{R}^{n}) = \left\{u\in H^{m}(\mathbb{R}^{n}): D^{J} u\in H^{\sigma}(\mathbb{R}^{n}) \mbox{\,\,for all\,\,}J \mbox{\,\,with\,\,} |J| = m\right\}
\end{eqnarray*}
and this is a Banach space with respect to the norm
\begin{eqnarray*}
\|u\|_{\gamma} := \|u\|_{H^{\gamma}(\mathbb{R}^{n})} = \left(\|u\|^2_{H^{m}(\mathbb{R}^{n})} + \displaystyle\sum_{|J| = m}\|D^{J} u\|^2_{H^{\sigma}(\mathbb{R}^{n})}\right)^{1/2}.
\end{eqnarray*}
Clearly, if $\gamma = m$ is an integer, the space $H^{\gamma}(\mathbb{R}^{n})$ coincides with the usual Sobolev space $H^{m}(\mathbb{R}^{n})$.

Theorem \ref{t:abstracttheorem} will be applied here to the following setting:
\begin{equation*}
E = D^{\gamma}(\mathbb{R}^n) : = \{u\in L^{\frac{2n}{n - 2\gamma}}(\mathbb{R}^n): \|(-\Delta)^{\gamma/2}u\|_{L^2(\mathbb{R}^n)} < \infty\}
\end{equation*}
\begin{equation*}
{\rm with\,\,norm\,\,}\|u\|_{D^{\gamma}(\mathbb{R}^n)} =  \|(-\Delta)^{\gamma/2}u\|_{L^2(\mathbb{R}^n)},
\end{equation*}
\begin{equation*}
f_0(u) = \frac{1}{2}\displaystyle\int_{\mathbb{R}^n}|(-\Delta)^{\frac{\gamma}{2}} u(x)|^2dx - \frac{n - 2\gamma}{2n}\displaystyle\int_{\mathbb{R}^n}{u_+}^{\frac{2n}{n - 2\gamma}}dx,
\end{equation*}
\begin{equation*}
G(u) = \frac{n - 2\gamma}{2n}\displaystyle\int_{\mathbb{R}^n}K(x) {u_+}^{\frac{2n}{n - 2\gamma}}dx,
\end{equation*}
where $u_+ = \max\{u, 0\}$.
In the reminder of this paper, we will show that conditions (i)-(iii) and (\ref{e:degreformula3}) hold true.

\section{The unperturbed problem.}
In this section, we consider the unperturbed equation,
\begin{equation}\label{e:unperturbed}
(-\Delta)^\gamma u = u^{\frac{n + 2\gamma}{n - 2\gamma}}, \quad u > 0{\rm\,\,\,on\,\,\,}\mathbb{R}^n.
\end{equation}

This equation arises from the Hardy-Littlewood-Sobolev inequality, which states the existence of a
positive number $S$ such that for all $u\in C^{\infty}_0(\mathbb{R}^n)$, one has
\begin{equation}\label{e:hardylittlewoodsobolevinequality}
S\|u\|_{L^{2^*}(\mathbb{R}^n)}\leq \|(-\Delta)^{\gamma/2}u\|_{L^2(\mathbb{R}^n)},
\end{equation}
where $2^* = \frac{2n}{n - 2\gamma}$. The optimal constant in (\ref{e:hardylittlewoodsobolevinequality}) was first found by Lieb in \cite{Lieb1983}. Lieb established that the extremals corresponding precisely to functions of form
\begin{equation}\label{e:classification}
z_{\mu, \xi}(x) = \alpha \left(\frac{\mu}{\mu^2 + |x - \xi|^2}\right)^{\frac{n - 2\gamma}{2}},\quad \alpha > 0,\quad \mu > 0,\quad \xi\in\mathbb{R}^n,
\end{equation}
which for suitable $\alpha = \alpha_{n, \gamma}$ solves (\ref{e:unperturbed}). Alternative proofs for this result were established in \cite{FrankLieb2010},
\cite{FrankLieb2012} and \cite{CarlenLoss1990}. Indeed, under some decay assumptions, (\ref{e:classification}) are the only solutions for (\ref{e:unperturbed}). In \cite{ChenWenxiongLicongming2006CPAM}, Chen, Li and Ou proved that every positive solution $u\in L^{\frac{2n}{n - 2\gamma}}_{loc}(\mathbb{R}^n)$ to (\ref{e:unperturbed}) are of the form (\ref{e:classification}). Similar results were also proved in \cite{LiYanyan2004} and \cite{LiZhu1995}.

To apply the perturbation argument, some type of non-degeneracy condition is
required. Checking this condition will be much harder than in the classical case, since it is
usually very difficult to compute explicitly fractional derivatives and singular integrals.

The linearized equation of (\ref{e:unperturbed}) is
\begin{equation}\label{e:linearized}
(-\Delta)^\gamma \phi = pu^{p - 1}\phi,
\end{equation}
where $p = \frac{n + 2\gamma}{n - 2\gamma}$.
In \cite{DavilaDelpinoSire2013}, the following non-degeneracy result was proved.

\begin{theorem}\label{l:nondegeneracy}(\cite{DavilaDelpinoSire2013})
The solutions
\begin{equation*}
z_{\mu, \xi}(x) = \alpha_{n, \gamma} \left(\frac{\mu}{\mu^2 + |x - \xi|^2}\right)^{\frac{n - 2\gamma}{2}}, \quad \mu > 0,\quad \xi\in\mathbb{R}^n,
\end{equation*}
of (\ref{e:unperturbed}) is nondegenerate in the sense that all bounded solutions of equation (\ref{e:linearized}) with $u = z_{\mu, \xi}$
are linear combination of the functions
\begin{equation*}
\partial_\mu z_{\mu, \xi},\,\,\,\partial_{\xi_i}z_{\mu, \xi},\,\,\,1\leq i \leq n.
\end{equation*}
\end{theorem}

For the fractional nonlinear Schr\"{o}dinger equation
\begin{equation}\label{e:schrodinger}
(-\Delta)^\gamma u + u = |u|^{\alpha}u\quad\mbox{on}\quad \mathbb{R}^n,
\end{equation}
where $\alpha\in(0, 4\gamma/(n - 2\gamma))$ when $ 0 < \gamma < n/2$ and $\alpha\in(0, \infty)$ when $ \gamma \geq n/2$,
Frank, Lenzmann and Silvestre proved the nondegeneracy of solutions of (\ref{e:schrodinger}), see \cite{Frank&LenzmannActaMath} and \cite{FrankLenzmann&Silvestre}.

Let $Z$ be defined as
\begin{equation*}
Z = \{z_{\mu, \xi}: \mu\in \mathbb{R}_+, \xi\in\mathbb{R}^n\}.
\end{equation*}
All the functions in $Z$ are critical points of $f_0$ and $Z$ is a $n + 1$ dimensional manifold in $E = D^{\gamma}(\mathbb{R}^n)$ parameterized by the map $\alpha: A = \mathbb{R}_+\times \mathbb{R}^n\to E$, $\alpha(\mu, \xi) = z_{\mu, \xi}$. So condition (i) holds.

It is easy to see that $D^2f_0(z)$ is a Fredholm operator with index zero for all $z\in Z$. So condition (ii) holds.

From Theorem \ref{l:nondegeneracy}, condition (iii) holds.

\section{Proof of the main results.}
In this section, we will show that under the hypotheses of Theorem \ref{t:main}, (\ref{e:degreformula3}) is satisfied for some suitable set $\Omega$ and complete the proof.

From the definition of $\Gamma$ (see Section 3), we have
\begin{eqnarray*}
\Gamma(\mu, \xi) &=& \frac{1}{p + 1}\displaystyle\int_{\mathbb{R}^n}K(x)z_{\mu, \xi}^{p + 1}(x)dx\\
&=& \frac{1}{p + 1}\displaystyle\int_{\mathbb{R}^n}K(\mu y + \xi)z_0^{p + 1}(y)dy.
\end{eqnarray*}
Here $z_0 := z_{1, 0} = \alpha_{n, \gamma}(\frac{1}{1 + |x|^2})^{(n - 2\gamma)/2}$.

\begin{lemma}
$\Gamma$ can be extended to $\mathbb{R}^{n + 1}$ as a $C^1$ function.
\end{lemma}
\begin{proof}
First, we extend $\Gamma$ to the hyperplane $\{(0, \xi):\xi\in \mathbb{R}^n\}$ by
\begin{equation*}
\Gamma(0, \xi) = \frac{1}{p + 1}\displaystyle\int_{\mathbb{R}^n}K(\xi)z_0(y)^{p + 1}dy \equiv c_0 K(\xi),
\end{equation*}
where $c_0 = \frac{1}{p + 1}\displaystyle\int_{\mathbb{R}^n}z_0(y)^{p + 1}dy$. A direct computation tells
\begin{equation*}
D_{\mu}\Gamma(\mu, \xi) = \frac{1}{p + 1}\displaystyle\int_{\mathbb{R}^n}\langle K'(\mu y + \xi), y\rangle z^{p + 1}_0(y)dy.
\end{equation*}
Since
$$\int_{\mathbb{R}^n}y_iz^{p + 1}_0(y)dy = 0 {\rm\,\, for\,\, all\,\,} i = 1, \cdot\cdot\cdot, n,$$ one has
\begin{equation}\label{e:reductionfunction}
D_{\mu}\Gamma(0, \xi) = 0.
\end{equation}
So by symmetry, $\Gamma$ can be extended to $\mathbb{R}^{n + 1}$ as a $C^1$ function.
\end{proof}
We will still denote this extended function as $\Gamma$.

\begin{remark}\label{r:equivalence}
From (\ref{e:reductionfunction}), it holds that
\begin{equation*}\label{e:equivalence}
\xi\in {\rm Crit}(K) \Leftrightarrow (0, \xi)\in {\rm Crit}(\Gamma).
\end{equation*}
\end{remark}

To prove (\ref{e:degreformula3}), we need the following lemmas.
\begin{lemma}\label{e:gamma}
Suppose that $K$ satisfies condition (K1).
Then there exists $R > 0$ such that
\begin{equation*}
\langle \Gamma'(\mu, \xi), (\mu, \xi)\rangle < 0 \quad\rm{for\,\, all\,\,} |\mu| + |\xi| \geq R.
\end{equation*}
\end{lemma}

\begin{lemma}\label{l:lemma1}
Suppose that $K\in L^{\infty}(\mathbb{R}^n)\cap C^2(\mathbb{R}^n)$. Then
\begin{equation*}\label{e:secondderivative1}
D^2_{\mu, \xi_i}\Gamma(0, \xi) = 0,\quad \forall i = 1,\cdot\cdot\cdot,n,
\end{equation*}
\begin{equation*}\label{e:secondderivative2}
D^2_{\mu, \mu}\Gamma(0, \xi) = c_1\Delta K(\xi),
\end{equation*}
where
\begin{equation*}
c_1 = \frac{1}{N(p + 1)}\displaystyle\int_{\mathbb{R}^n}|y|^2z_0^{p + 1}dy.
\end{equation*}
\end{lemma}

\begin{lemma}\label{l:lemma2}
Suppose (K3) hold.
Then
\begin{equation}\label{e:limite}
\displaystyle\lim_{\mu\to 0^+}\frac{\Gamma(\mu,\xi) - \Gamma(0, \xi)}{\mu^\beta} = A_\xi.
\end{equation}
\end{lemma}

\begin{lemma}\label{l:lemma3}
Suppose (K2) and (K3) hold.
Then $q = (0, \xi)$ is an isolated critical point of $\Gamma$ and the following results hold,
\begin{eqnarray*}
A_\xi > 0\Rightarrow \rm{deg_{loc}}(\Gamma', q) = \rm{deg_{loc}}(K', \xi)\\
A_\xi < 0\Rightarrow \rm{deg_{loc}}(\Gamma', q) = -\rm{deg_{loc}}(K', \xi).
\end{eqnarray*}
\end{lemma}

The proof of Lemma \ref{e:gamma}, \ref{l:lemma1}, \ref{l:lemma2} and \ref{l:lemma3} is similar to \cite{AmbrosettiGarciaPeral},
so we omit it.

Let $R\geq \eta$, where $\eta$ is the constant in condition (K1). Set $B_R^n = \{x\in \mathbb{R}^n: |x| < R\}$. From (K1), we have
$$\langle K'(x), x\rangle < 0 {\rm \,\,for\,\, all\,\,} |x| = R.$$
Thus we obtain
\begin{equation*}
{\rm deg}(K', B_R^n, 0) = (-1)^n.
\end{equation*}
By the properties of the topological degree, the above formula is equivalent to
\begin{equation*}
\displaystyle\sum_{\xi\in {\rm Crit}(K)} {\rm deg_{loc}}(K', \xi) = (-1)^n.
\end{equation*}
Since $A_\xi \neq 0$ for all $\xi\in {\rm Crit}(K)$, we also have
\begin{equation}\label{e:degreeformula}
\displaystyle\sum_{A_\xi > 0}{\rm deg_{loc}}(K', \xi) + \displaystyle\sum_{A_\xi < 0}{\rm deg_{loc}}(K', \xi)= (-1)^n.
\end{equation}
Let $\vartheta^+ = \{(\mu, \xi)\in {\rm Crit}(\Gamma):\mu > 0\}$.
Then we have
\begin{lemma}
$\vartheta^+$ is a (possibly empty) compact set.
\end{lemma}
\begin{proof}
From Lemma \ref{e:gamma}, we have that $\vartheta^+$ is bounded. Since (\ref{e:limite}) holds, $\Gamma$ does not have critical points in a sufficiently small neighborhood of $(0,\xi)$. Therefore, $\vartheta^+$ is compact.
\end{proof}
By the evenness of $\Gamma$ is in $\mu$, $\vartheta^- = \{(-\mu, \xi): (\mu, \xi)\in \vartheta^+\}$ also consists of critical points of $\Gamma$.

\begin{lemma}
There is a bounded open set $\Omega \subset (0, \infty)\times \mathbb{R}^n$ with $\vartheta^+ \subset \Omega$ such that
\begin{equation*}
{\rm deg}(\Gamma', \Omega, 0)\neq 0.
\end{equation*}
\end{lemma}
\begin{proof}
By Lemma \ref{e:gamma} it holds that
\begin{equation*}
{\rm deg}(\Gamma', B_R^{n + 1}, 0) = (-1)^{n + 1}.
\end{equation*}
Suppose the conclusion is not true, let $\Omega$ be an open bounded set such that
$$\vartheta^+ \subset\Omega\subset (0, \infty)\times \mathbb{R}^n$$
and
$${\rm deg}(\Gamma', \Omega, 0) = 0.$$
Let $\Omega^- = \{(-\mu, \xi):(\mu, \xi)\in \Omega\}$ and $\Omega' = \Omega \cup \Omega^-$. Then
$${\rm deg}(\Gamma', \Omega^-, 0) = 0.$$
Hence
\begin{equation*}
{\rm deg}(\Gamma', B_R^{n + 1}\setminus \overline{\Omega'}, 0) = (-1)^{n + 1}.
\end{equation*}
By remark \ref{r:equivalence}, any $q\in {\rm Crit}(\Gamma)\setminus (\vartheta^+ \cup \vartheta^-)$ has the form $q = (0, \xi)$ for some $\xi\in {\rm Crit}(K)$. Then by Lemma \ref{l:lemma3},
\begin{eqnarray*}
(-1)^{n + 1} = {\rm deg}(\Gamma', B_R^{n + 1}\setminus \bar{\Omega'}, 0) &=& \sum {\rm deg_{loc}}(\Gamma', q)\\
&=& \displaystyle\sum_{A_\xi > 0}{\rm deg_{loc}}(K', \xi) - \displaystyle\sum_{A_\xi < 0}{\rm deg_{loc}}(K', \xi).
\end{eqnarray*}
From (\ref{e:degreeformula}), one obtains
\begin{equation*}
\displaystyle\sum_{A_\xi < 0}{\rm deg_{loc}}(K', \xi)  = (-1)^n,
\end{equation*}
a contradiction with (K4).
\end{proof}

Thus (\ref{e:degreformula3}) is satisfied.

{\bf Proof of Theorem \ref{t:main}}: {\bf Existence}: Since all the conditions of Theorem \ref{t:abstracttheorem} are satisfied, we find a critical point $u_\varepsilon\in D^{\gamma}(\mathbb{R}^n)$ of $f_\varepsilon$. It is a weak solution of problem (\ref{e:main}) in the sense that
\begin{equation*}
\displaystyle\int_{\mathbb{R}^n}(-\Delta)^{\frac{\gamma}{2}} u(x)(-\Delta)^{\frac{\gamma}{2}} \phi(x)dx =  \displaystyle\int_{\mathbb{R}^n}(1 + \varepsilon K(x)){u_+(x)}^{\frac{n + 2\gamma}{n - 2\gamma}}\phi(x)dx,
\end{equation*}
for any $\phi\in D^{\gamma}(\mathbb{R}^n)$.

{\bf Positivity}: For any $\phi\in C_0^{\infty}(\mathbb{R}^n)$, let
\begin{equation*}
\psi(x) = \displaystyle\int_{\mathbb{R}^n}\frac{\phi(y)}{|x - y|^{n - 2\gamma}}.
\end{equation*}
Then $(-\Delta)^{\gamma} \psi = \phi$ and so $\psi\in D^{\gamma}(\mathbb{R}^n)$. Therefore, we have
\begin{equation*}
\displaystyle\int_{\mathbb{R}^n}(-\Delta)^{\frac{\gamma}{2}} u(x)(-\Delta)^{\frac{\gamma}{2}} \psi(x)dx =  \displaystyle\int_{\mathbb{R}^n}(1 + \varepsilon K(x)){u_+(x)}^{\frac{n + 2\gamma}{n - 2\gamma}}\psi(x)dx.
\end{equation*}
Integration by parts of the left-hands side and exchanging the order of integration of the right-hand side yields
\begin{equation*}
\displaystyle\int_{\mathbb{R}^n}u(x)\phi(x)dx =  \displaystyle\int_{\mathbb{R}^n}\left\{\int_{\mathbb{R}^n} \frac{(1 + \varepsilon K(y))u_+(y)^{\frac{n + 2\gamma}{n - 2\gamma}}}{|x - y|^{n - 2\gamma}}dy\right\}\phi(x)dx.
\end{equation*}
Since $\phi$ is any function in $C_0^{\infty}(\mathbb{R}^n)$, we conclude that
\begin{equation*}
u(x) =  \int_{\mathbb{R}^n} \frac{(1 + \varepsilon K(y))u_+(y)^{\frac{n + 2\gamma}{n - 2\gamma}}}{|x - y|^{n - 2\gamma}}dy.
\end{equation*}
Thus $u \geq 0$ and it satisfies
\begin{equation}\label{e:integral}
u(x) =  \int_{\mathbb{R}^n} \frac{(1 + \varepsilon K(y))u(y)^{\frac{n + 2\gamma}{n - 2\gamma}}}{|x - y|^{n - 2\gamma}}dy.
\end{equation}
Since $u\neq 0$, from this formula, we have $u(x) > 0$.

{\bf Regularity}: Minor modification of the proof of Theorem 1.2 in \cite{LiYanyan2004} and from the integral formula (\ref{e:integral}) of $u$, we have $u \in C^{\infty}(\mathbb{R}^n)$.

Since $u\in D^{\gamma}(\mathbb{R}^n)$, $\lim_{|x|\to 0}u(x) = 0$. Thus this solution can be lifted to the sphere $\mathbb{S}^n$.

This completes the proof.

{\bf Proof of Corollary \ref{c:corollary}}:
In the case of (K5), let $\beta = 2$ and $Q_\xi = D^2_{ij}K(\xi)(x - \xi)^2$ in (K3). Then we have $A_\xi = c_1\Delta K(\xi) \neq 0$. So Theorem \ref{t:main} can be applied. In case of (K6), set
\begin{equation*}
A_\xi = \sum a_j(\xi)\cdot \displaystyle\int_{\mathbb{R}^n}|y_1|^\beta z_0^{p + 1}dy\neq 0
\end{equation*}
and Theorem \ref{t:main} can be applied.

\bibliography{mrabbrev,mlabbr2003-0,localbib}

\begin{thebibliography}{10}

\bibitem{AdamsSobolevSpace}
Robert~A. Adams.
\newblock {\em Sobolev spaces}.
\newblock Academic Press [A subsidiary of Harcourt Brace Jovanovich,
  Publishers], New York-London, 1975.
\newblock Pure and Applied Mathematics, Vol. 65.

\bibitem{AmbrosettiCotiZelatiEkeland1987}
A.~Ambrosetti, V.~Coti~Zelati, and I.~Ekeland.
\newblock Symmetry breaking in {H}amiltonian systems.
\newblock {\em J. Differential Equations}, 67(2):165--184, 1987.

\bibitem{AmbrosettiGarciaPeral}
A.~Ambrosetti, J.~Garcia~Azorero, and I.~Peral.
\newblock Perturbation of {$\Delta u+u^{(N+2)/(N-2)}=0$}, the scalar curvature
  problem in {${R}^N$}, and related topics.
\newblock {\em J. Funct. Anal.}, 165(1):117--149, 1999.

\bibitem{AmbrosettiBadiale1998}
Antonio Ambrosetti and Marino Badiale.
\newblock Homoclinics: {P}oincar\'e-{M}elnikov type results via a variational
  approach.
\newblock {\em Ann. Inst. H. Poincar\'e Anal. Non Lin\'eaire}, 15(2):233--252,
  1998.

\bibitem{AmbrosettiMalchiodi2006}
Antonio Ambrosetti and Andrea Malchiodi.
\newblock {\em Perturbation methods and semilinear elliptic problems on
  {${R}^n$}}, volume 240 of {\em Progress in Mathematics}.
\newblock Birkh\"auser Verlag, Basel, 2006.

\bibitem{BahriCoronJFA1991}
A.~Bahri and J.-M. Coron.
\newblock The scalar-curvature problem on the standard three-dimensional
  sphere.
\newblock {\em J. Funct. Anal.}, 95(1):106--172, 1991.

\bibitem{CaffarelliRoquejoffreSavin}
L.~Caffarelli, J.-M. Roquejoffre, and O.~Savin.
\newblock Nonlocal minimal surfaces.
\newblock {\em Comm. Pure Appl. Math.}, 63(9):1111--1144, 2010.

\bibitem{Caffarelli&Silvestre07}
Luis Caffarelli and Luis Silvestre.
\newblock An extension problem related to the fractional {L}aplacian.
\newblock {\em Comm. Partial Differential Equations}, 32(7-9):1245--1260, 2007.

\bibitem{CaffarelliValdinoci2011}
Luis Caffarelli and Enrico Valdinoci.
\newblock Uniform estimates and limiting arguments for nonlocal minimal
  surfaces.
\newblock {\em Calc. Var. Partial Differential Equations}, 41(1-2):203--240,
  2011.

\bibitem{CaffarelliSalsa&Silvestre}
Luis~A. Caffarelli, Sandro Salsa, and Luis Silvestre.
\newblock Regularity estimates for the solution and the free boundary of the
  obstacle problem for the fractional {L}aplacian.
\newblock {\em Invent. Math.}, 171(2):425--461, 2008.

\bibitem{CaffarelliVasseurDFDQGE2010}
Luis~A. Caffarelli and Alexis Vasseur.
\newblock Drift diffusion equations with fractional diffusion and the
  quasi-geostrophic equation.
\newblock {\em Ann. of Math. (2)}, 171(3):1903--1930, 2010.

\bibitem{CarlenLoss1990}
Eric~A. Carlen and Michael Loss.
\newblock Extremals of functionals with competing symmetries.
\newblock {\em J. Funct. Anal.}, 88(2):437--456, 1990.

\bibitem{ChangGonzalez2011}
Sun-Yung~Alice Chang and Mar{\'{\i}}a del~Mar Gonz{\'a}lez.
\newblock Fractional {L}aplacian in conformal geometry.
\newblock {\em Adv. Math.}, 226(2):1410--1432, 2011.

\bibitem{ChenWenxiongLicongming2006CPAM}
Wenxiong Chen, Congming Li, and Biao Ou.
\newblock Classification of solutions for an integral equation.
\newblock {\em Comm. Pure Appl. Math.}, 59(3):330--343, 2006.

\bibitem{NezzaPalatucci&Valdinoci}
Eleonora Di~Nezza, Giampiero Palatucci, and Enrico Valdinoci.
\newblock Hitchhiker's guide to the fractional {S}obolev spaces.
\newblock {\em Bull. Sci. Math.}, 136(5):521--573, 2012.

\bibitem{DjadliHebeyLedoux2000}
Zindine Djadli, Emmanuel Hebey, and Michel Ledoux.
\newblock Paneitz-type operators and applications.
\newblock {\em Duke Math. J.}, 104(1):129--169, 2000.

\bibitem{DavilaDelpinoSire2013}
Juan. D¨¢vila, Manuel. Del~Pino, and Yannick. Sire.
\newblock Nondegeneracy of the bubble in the critical case for nonlocal
  equations.
\newblock {\em Proc. Amer. Math. Soc.}, (to appear), 2013.

\bibitem{FeffermanGraham2012}
Charles Fefferman and C.~Robin Graham.
\newblock {\em The ambient metric}, volume 178 of {\em Annals of Mathematics
  Studies}.
\newblock Princeton University Press, Princeton, NJ, 2012.

\bibitem{Frank&LenzmannActaMath}
Rupert Frank and Enno Lenzmann.
\newblock Uniqueness of non-linear ground states for fractional {L}aplacians in
  {R}.
\newblock {\em Acta Math.£¬}, 210(2):261--318, 2013.

\bibitem{FrankLenzmann&Silvestre}
Rupert Frank, Enno Lenzmann, and Luis Silvestre.
\newblock Uniqueness of radial solutions for the fractional {L}aplacian.
\newblock {\em arXiv:1302.2652}.

\bibitem{FrankLieb2010}
Rupert~L. Frank and Elliott~H. Lieb.
\newblock Inversion positivity and the sharp {H}ardy-{L}ittlewood-{S}obolev
  inequality.
\newblock {\em Calc. Var. Partial Differential Equations}, 39(1-2):85--99,
  2010.

\bibitem{FrankLieb2012}
Rupert~L. Frank and Elliott~H. Lieb.
\newblock A new, rearrangement-free proof of the sharp
  {H}ardy-{L}ittlewood-{S}obolev inequality.
\newblock In {\em Spectral theory, function spaces and inequalities}, volume
  219 of {\em Oper. Theory Adv. Appl.}, pages 55--67. Birkh\"auser/Springer
  Basel AG, Basel, 2012.

\bibitem{GonzalezJGA2012}
Maria del~Mar Gonz{\'a}lez, Rafe Mazzeo, and Yannick Sire.
\newblock Singular solutions of fractional order conformal {L}aplacians.
\newblock {\em J. Geom. Anal.}, 22(3):845--863, 2012.

\bibitem{GonzalezAPDE2013}
Maria del~Mar Gonz{\'a}lez and jie Qing.
\newblock Fractional conformal {L}aplacians and fractional {Y}amabe problems.
\newblock {\em Analysis and PDE}, page to appear.

\bibitem{GrahamCRobinJLMS1992}
C.~Robin Graham.
\newblock Conformally invariant powers of the {L}aplacian. {II}.
  {N}onexistence.
\newblock {\em J. London Math. Soc. (2)}, 46(3):566--576, 1992.

\bibitem{GrahamJenneSparling1992}
C.~Robin Graham, Ralph Jenne, Lionel~J. Mason, and George A.~J. Sparling.
\newblock Conformally invariant powers of the {L}aplacian. {I}. {E}xistence.
\newblock {\em J. London Math. Soc. (2)}, 46(3):557--565, 1992.

\bibitem{GramZworski2003}
C.~Robin Graham and Maciej Zworski.
\newblock Scattering matrix in conformal geometry.
\newblock {\em Invent. Math.}, 152(1):89--118, 2003.

\bibitem{JinLiyanyanXiongNirenbergproblemBlowupanalysis}
Tianling Jin, yanyan Li, and Jingang Xiong.
\newblock On a fractional nirenberg problem, part i: blow up analysis and
  compactness of solutions.
\newblock {\em J. Eur. Math. Soc.}, page to appear.

\bibitem{JinXiongfYfsapplications2013}
Tianling Jin and Jingang Xiong.
\newblock A fractional {Y}amabe flow and some applications.
\newblock {\em Journal fur die reine und angewandte Mathematik}, page to
  appear.

\bibitem{Juhl2010}
Andreas Juhl.
\newblock On conformally covariant powers of the {L}aplacian.
\newblock {\em arXiv:0905.3992}.

\bibitem{Juhl2009}
Andreas Juhl.
\newblock {\em Families of conformally covariant differential operators,
  {$Q$}-curvature and holography}, volume 275 of {\em Progress in Mathematics}.
\newblock Birkh\"auser Verlag, Basel, 2009.

\bibitem{LiyanyanPSCRPJDE1995}
Yan~Yan Li.
\newblock Prescribing scalar curvature on {$S^n$} and related problems. {I}.
\newblock {\em J. Differential Equations}, 120(2):319--410, 1995.

\bibitem{LiYanyan2004}
Yan~Yan Li.
\newblock Remark on some conformally invariant integral equations: the method
  of moving spheres.
\newblock {\em J. Eur. Math. Soc. (JEMS)}, 6(2):153--180, 2004.

\bibitem{LiYanyanPSCRPCPAM1996}
Yanyan Li.
\newblock Prescribing scalar curvature on {$S^n$} and related problems. {II}.
  {E}xistence and compactness.
\newblock {\em Comm. Pure Appl. Math.}, 49(6):541--597, 1996.

\bibitem{LiZhu1995}
Yanyan Li and Meijun Zhu.
\newblock Uniqueness theorems through the method of moving spheres.
\newblock {\em Duke Math. J.}, 80(2):383--417, 1995.

\bibitem{Lieb1983}
Elliott~H. Lieb.
\newblock Sharp constants in the {H}ardy-{L}ittlewood-{S}obolev and related
  inequalities.
\newblock {\em Ann. of Math. (2)}, 118(2):349--374, 1983.

\bibitem{MalchiodiUguzzoni2002}
Andrea Malchiodi and Francesco Uguzzoni.
\newblock A perturbation result for the {W}ebster scalar curvature problem on
  the {CR} sphere.
\newblock {\em J. Math. Pures Appl. (9)}, 81(10):983--997, 2002.

\bibitem{MazzeoMelrose1987}
Rafe~R. Mazzeo and Richard~B. Melrose.
\newblock Meromorphic extension of the resolvent on complete spaces with
  asymptotically constant negative curvature.
\newblock {\em J. Funct. Anal.}, 75(2):260--310, 1987.

\bibitem{Paneitz2008}
Stephen~M. Paneitz.
\newblock A quartic conformally covariant differential operator for arbitrary
  pseudo-{R}iemannian manifolds (summary).
\newblock {\em SIGMA Symmetry Integrability Geom. Methods Appl.}, 4:Paper 036,
  3, 2008.

\bibitem{Peterson2000}
Lawrence~J. Peterson.
\newblock Conformally covariant pseudo-differential operators.
\newblock {\em Differential Geom. Appl.}, 13(2):197--211, 2000.

\bibitem{QingRaske2006}
Jie Qing and David Raske.
\newblock On positive solutions to semilinear conformally invariant equations
  on locally conformally flat manifolds.
\newblock {\em Int. Math. Res. Not.}, pages Art. ID 94172, 20, 2006.

\bibitem{Shubin01}
M.~A. Shubin.
\newblock {\em Pseudodifferential operators and spectral theory}.
\newblock Springer-Verlag, Berlin, second edition, 2001.
\newblock Translated from the 1978 Russian original by Stig I. Andersson.

\end{thebibliography}
\bibliographystyle{plain}
\end{document}